\documentclass[12pt]{amsart}
\usepackage{amsmath,amssymb,amscd,amsxtra}
\usepackage{graphics,graphicx}
\usepackage{latexsym}
\usepackage{mathtools}
\usepackage{cases}
\newtheorem{thm}{Theorem}
\newtheorem{lemma}{Lemma}
\newtheorem{prop}{Proposition}
\newtheorem{cor}{Corollary}
\theoremstyle{remark}
\newtheorem{rem}{Remark}
\theoremstyle{definition}

\newcommand{\R}{\mathbb{R}}

\newcommand{\Z}{\mathbb{Z}}

\newcommand{\cG}{\mathcal{G}}

\newcommand{\ip}[2]{\langle#1,#2\rangle}

\newcommand{\e}{\mathrm{e}}

\title{Frame sets for generalized $B$-splines}

\author{A.~Ganiou D.~Atindehou}
\address{Institut de Math\'ematiques et de Sciences Physiques (IMSP), 01 BP 613, Porto-Novo, 
B\'enin}
\email{ganiouatindehou@gmail.com}

\author{Yebeni B.~Kouagou}
\address{Institut de Math\'ematiques et de Sciences Physiques (IMSP), 01 BP 613, Porto-Novo, 
B\'enin}
\email{yebeni.kouagou@gmail.com}

\author{Kasso A.~Okoudjou}
\address{Department of Mathematics, University
of Maryland, College Park, MD 20742}
\email{kasso@math.umd.edu}


\date{\today}


\begin{document}

\begin{abstract} The frame set of a function  $g\in L^2(\R)$ is the subset of all parameters $(a, b)\in \R^2_+$ for which the time-frequency shifts of $g$ along $a\Z\times b\Z$ form a Gabor frame for $L^2(\R).$ In this paper,  we investigate the frame set of a class of functions that we call \emph{generalized $B-$splines} and which  includes the $B-$splines. In particular, we add many new points to the frame sets of these functions. In the process, we generalize and unify some recent results on the frame sets for this class of functions.
\end{abstract}

\maketitle \pagestyle{myheadings} \thispagestyle{plain}
\markboth{A. G. D.  Atindehou, Y. B. Kouagou, And  K. A. Okoudjou}{FRAME SETS FOR GENERALIZED $B$-SPLINES}


\section{Introduction}\label{sec1}

  Given $g\in L^2(\R)$ and $a, b>0$, the collection of functions 
\begin{equation*}
  \cG(g,a, b)=\left\{M_{\ell b}T_{ka}g =\e^{2\pi i\ell b\cdot}g(\cdot-ka):(\ell,k)\in \mathbb{Z}^2\right\}
\end{equation*} 
is a {\em Gabor frame} for $L^2(\R)$ if there exist  $A, B>0$ such that 
\begin{equation*}
 \label{frame inequality}
 A \|f\|^2_2\leq\sum_{\ell,k\in \Z}  |\ip{f}{M_{\ell b}T_{ka}g}|^2\leq B\|f\|^2_2,
 \end{equation*}
 for all $f\in L^2(\R)$.  It follows that there exists a function $h\in L^2(\R)$ such that for every $f\in L^2$ we have 
$$
f=\sum_{k, \ell \in \Z}\ip{f}{M_{b\ell}T_{ka}h} M_{b\ell}T_{ka}g=\sum_{k, \ell \in \Z}\ip{f}{M_{b\ell}T_{ka}g} M_{b\ell}T_{ka}h.
$$ 
For more details on Gabor analysis we refer to \cite{Ole4, Groc2001}.

For  $g\in L^2(\R)$, finding  the set of all points $(a, b)\in \R_+^2$ such that $ \cG(g,a, b)$ is a Gabor frame for $L^2(\R)$ remains one  of the field's fundamental yet mostly unresolved question. 
The set of all such parameters is customarily referred to as  the {\em frame set } of $g$ and given by
$$
\mathcal{F}(g)=\left\{(a,b)\in \R_+^2:\, \cG(g,a, b)\ \mbox{is a frame}\right\}.
$$

For a recent survey of the structure of $\mathcal{F}(g)$ we refer to \cite{Groch}. A particular property of the frame set of $g$ in the modulation space $M^1(\R)$ (\cite{Groc2001}), was obtained by Feichtinger and Kaiblinger who  proved that in this case, $\mathcal{F}(g)$ is an open subset of $\R_+^2 $ \cite{FeiKai04}. However, the complete characterization of $\mathcal{F}(g)$ is only known in the following cases. 

\begin{enumerate}
\item[(a)] The Gaussian $g(x)=e^{-\pi x^2}$, \cite{Lyuba, Seip, SeiWal}.
\item[(b)] The hyperbolic secant $g(x)=\tfrac{1}{\cosh x}$, \cite{JanStroh}.
\item[(c)] The one-sided and two-sided exponentials, 
$g(x)=\e^{-x}\chi_{[0,+\infty]}(x)$ and \mbox{$g(x)=\e^{-|x|}$}, \cite{Jan, Janss}.
\item[(d)] The class of totally positive functions of finite type, \cite{Grosto}.
\item[(e)] The class of totally positive functions of Gaussian type, \cite{GrRoSt}.
\item[(f)] The characteristic function of an interval, i.e., $g(x)=\chi_{[0,c]}(x), c>0$, \cite{DaiSun, GuHan, Jans}.
\end{enumerate}

Recent progress has been made in  characterizing  the frame set for $B-$splines, i.e., $$g_1(x)= \chi_{[-1/2, 1/2]},\quad \text{and}\quad g_N(x)=g_1\ast g_{N-1}(x)\quad \text{for}\quad N\geq 2 $$ \cite{Olehon1, Olehon, KarGro, Kloosto, Lemniel}. 
For $N\geq 2$, $g_N$ belongs to the modulation space $M^1$ hence  $\mathcal{F}(g_N)$ is open. Nonetheless, finding $\mathcal{F}(g_N)$ for $N\geq 2$ is listed as one of the six problems in frame theory \cite{Ole1}.

By combining some of the aforementioned  results, we know that the region 
$$
((0, N/2]\times (0, 4/(N+3a)] ) \cup ([N/2, N]\times (0, 1/a))
$$ 
is included  in $\mathcal{F}(g_N)$.

\subsection{Our contributions} 
The main contribution  of this paper establishes that the region  $ [N/3, N/2]\times [4/(N+3a), 2/N)$  is contained in $\mathcal{F}(g_N)$. Consequently, the connected set 
 
$$
E=\big\{(a,b)\in \R_+^2:\,  ab<1,\, 0<a<N, \, 0<b\leq \max(\frac{2}{N}, \tfrac{4}{N+3a})\big\}
$$ 
is included in $\mathcal{F}(g_N)$. We refer to Figure~\ref{fig:figure1} for an illustration of the known results as well as our new results for $N=2$. In fact, we establish our results for the class of \emph{generalized $B-$splines}, $V_{N,a}$ introduced in \cite{Olehon} and given by

\begin{equation*}
 V_{N,a}:=\left\{g\in C(\mathbb{R}):\,  \mbox{supp}\ g=\left[-\frac{N}{2}, \frac{N}{2}\right], g\ \mbox{is real-valued and satisfies}\ (A1)-(A3)\right\}
\end{equation*}
where\newline

\noindent $(A1)$ $g$ is symmetric around the origin;\newline
\noindent $(A2)$ $g$ is strictly increasing on $\left[-\frac{N}{2}, 0\right]$;\newline
\noindent $(A3)$ If $a<\frac{N}{3}$, then $\bigtriangleup^2_ag(x)\geq 0,x\in\left[-\frac{N}{2},-\frac{N}{4}+\frac{3a}{4}\right]$, and  if $a\geq\frac{N}{3}$, then
\mbox{$\bigtriangleup^2_ag(x)\geq0$},
\mbox{$x\in \left[-\frac{N}{2},0\right]\bigcup\left\{-\frac{N}{4}+\frac{3a}{4}\right\}$}, where
$$ 
\bigtriangleup^2_ag(x)=g(x)-2g(x-a)+g(x-2a).
$$ 
We point out that the $B-$spline $g_N$ belongs to $\displaystyle{\bigcap_{0<a<N}}V_{N,a}$ for all $N\geq 2$, and we refer to  \cite[Section 3]{Olehon} for more examples of functions in $V_{N,a}$.

Before stating our main result, we first recall the following well-known facts on Gabor frames generated by continuous compactly supported functions, which will be the basis of our work. We refer to \cite{Ole4} for  proofs.

\begin{prop}\label{known}
 Let $N\geq 1$, and assume that $g:\mathbb{R}\rightarrow\mathbb{C}$ is a continuous function with $\text{supp}\ g\subseteq[-\frac{N}{2},\frac{N}{2}]$. Then the following holds:
 \begin{enumerate}
\item If $\cG(g,a, b)$ is a frame, then $ab<1$ and $a<N$.
\item\cite{Lemniel} Assume that $0<a<N$, $0<b\leq\frac{2}{N+a}$ and $\inf_{x\in[-\frac{a}{2},\frac{a}{2}]}|g(x)|>0$. Then $\cG(g,a, b)$ is a frame, and there is a unique 
dual $h\in L^2(\mathbb{R})$ such that \mbox{supp$h\subseteq\left[-\frac{a}{2},\frac{a}{2}\right]$}.
\item \cite{Olehon1} Assume that $\frac{N}{2}\leq a<N$ and $0<b<\frac{1}{a}$. If $g(x)>0,x\in]-\frac{N}{2},\frac{N}{2}[$, then $\cG(g,a, b)$ is a frame.
\item \cite{Olehon} Suppose that $0<a<N$, $\frac{2}{N+a}<b\leq\frac{4}{N+3a}$ and $g\in V_{N,a}$. Then $\cG(g,a, b)$ is a frame, and there is a unique 
dual $h\in L^2(\mathbb{R})$ such that
supp$h\subseteq\left[-\frac{3a}{2},\frac{3a}{2}\right]$.
 \end{enumerate}
\end{prop}

To prove our results, we use the following partition of the subset $\{ (a, b) \in E, 0 < a < \frac{N}{2}\}$ of $E$.  It seems that this partitioning method could be used to find more points in the frame set of functions in $V_{N,a}$. In fact, in a forthcoming paper \cite{AtiKouOko2}, we expand our method to add many new points to the frame set of the $2$-spline, $g_2$. 

 For $N\geq 2$ and given $g\in V_{N,a}$, our method consists in proving that  for each $m\geq 2$, $T_m\subset \mathcal{F}(g)$ where 

\begin{equation}
\hspace{-2mm}
T_m:=\left\{(a,b)\in\R^2_+:a\in \left(\frac{N(m-2)}{2m-3},\frac{N}{2}\right), b\in \left(\frac{2(m-1)}{N+(2m-3)a},\frac{2m}{N+(2m-1)a} \right], b<\frac{2}{N}\right\},
\end{equation} and 
$$T_1=\left\{(a,b)\in\R^2_+:a\in \left(0,\frac{N}{2}\right), b\in \left(0,\frac{2}{N+a} \right], b<\frac{2}{N}\right\}.$$

For the windows in $V_{N,a}$, the set $T_1$ was already investigated in \cite[Theorem 2]{Lemniel}, while $T_2$ was investigated in \cite[Theorem 1.2]{Olehon}. In both cases, it was shown that there exists a compactly supported dual window. In this paper we prove that for all $m\geq 3$, $T_m\subset \mathcal{F}(g)$ which implies that 
 
 \begin{equation*}
  T:=\bigcup^\infty_{m=1}T_m\subset  \mathcal{F}(g).
 \end{equation*}

More precisely, the following result, will be proved in Section~\ref{sec2} after we establish a number of technical results. 

\begin{thm}\label{th:main1} Given $a>0$ and $N\geq 2$, suppose that $g \in V_{N,a}$. For $m\geq3$, let $(a,b)\in T_m$ . Then the Gabor system $\cG(g,a, b)$ is a  frame for $L^2(\mathbb{R})$, and there is
a unique dual window $h \in L^2(\mathbb{R})$ such that ${\text supp}\,  h\subseteq [-\frac{2m-1}{2}a,\frac{2m-1}{2}a]$.
\end{thm}

Before proving Theorem~\ref{th:main1} in Section~\ref{sec2} we establish a number of technical results. The key technical result is Corollary~\ref{cor:detgm} in which we show that a certain tri-diagonal matrix is invertible by computing its determinant. As will be apparent from our proofs, this framework generalizes and unifies \cite[Theorem 1.2]{Olehon} and \cite[Theorem 2]{Lemniel}. In addition, our result is established by showing that for $a>0,$ $N\geq 2$,  $g\in V_{N,a}$, and $m\geq 3$, there exists a unique bounded compactly supported  dual window $h:=h_{a, N, m}$. Note however, that in contrast to $g$, this dual is discontinuous, and hence does not belong to the modulation space $M^1(\R)$. In particular, this indicates that the size of the support of the dual increases as the parameters $a, b$ approach the hyperbola $b=1/a$. For the type of techniques we develop in the sequel to be extended to other values of $a $ and $b$, it seems that a better understanding of the support of the plausible dual frame is needed.  It would be interesting to know whether or not for $g\in V_{N,a}$, and $(a, b)\in \mathcal{F}(g)$, there exist  compactly supported dual windows $h$. To the best of our knowledge this question has not been fully investigated. For more on the support properties of dual frames, we refer to \cite{Chrikim1, Olehon, Laug} and the references therein.

An immediate consequence of Theorem~\ref{th:main1} is.

\begin{cor}\label{cor:main1} Given $a>0$ and $N\geq 2$, suppose that $g \in V_{N,a}$.
 If $(a,b)\in T$, then the Gabor system $\cG(g,a, b)$ is a  frame for $L^2(\mathbb{R})$.
\end{cor}

\begin{rem}
Observe that in Theorem~\ref{th:main1} we only consider  $0<a<N/2$. However, our result extends to the regime $N/2\leq a<N$, which has already been considered. So we choose not to reprove the result in this case.   

Figure~\ref{fig:figure1} is a continuation of \cite[Figure 1]{Lemniel}. It  illustrates  and compares Theorem~\ref{th:main1} to Proposition~\ref{known} and the results in \cite{Olehon, Lemniel}.
\end{rem}

\begin{figure}[!h]
\hspace{-2cm}
\includegraphics[scale=0.35]{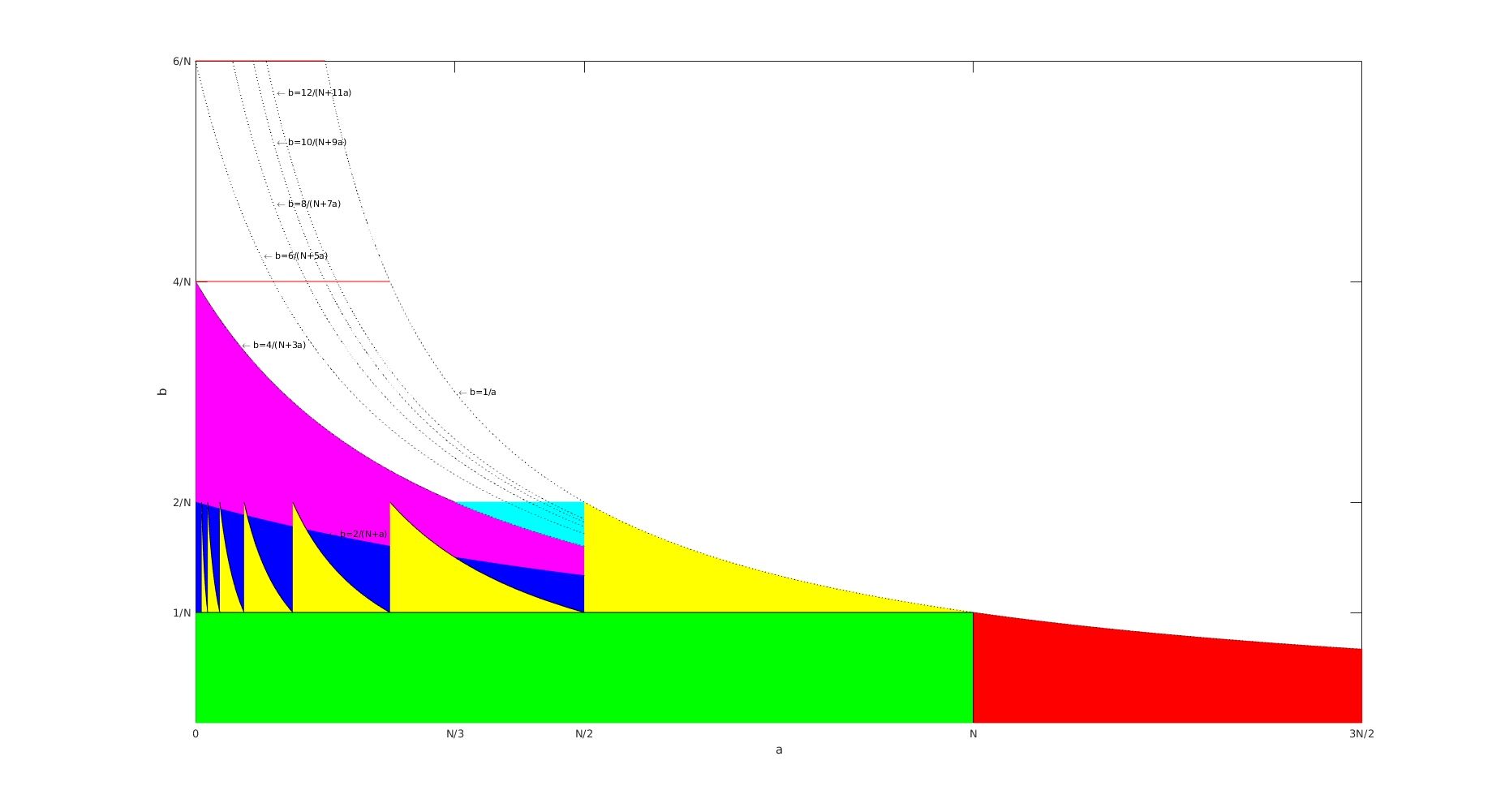}
\caption{ A sketch of  $\mathcal{F}(g_N)$ for $N=2$. The red region contains points $(a,b)$ for which $\cG\left(g_2,a,b\right)$ is not a frame. All other colors indicate the frame property. The green region is the classical: "painless expansions" \cite{DaGroMe}, and the yellow region is the result from \cite{Olehon1}. The blue and the magenta regions are respectively from \cite{Lemniel} and \cite{Olehon}. The cyan region is 
the result in Theorem~\ref{th:main1}.}
 \label{fig:figure1}
\end{figure}


\section{$T$ as a subset of the frame set for functions in $V_{N,a}$}\label{sec2}

Let $a>0, N\geq 2$ and $m\geq3$. For a function $g\in V_{N,a}$ and $(a,b)\in T_m$, we prove that $\cG(g,a,b)$ is a frame for $L^2(\R)$ by constructing a (unique) function $h\in L^2(\R)$ such that 
$\cG(g,a,b)$ and $\cG(h,a,b)$ are dual frames. This is achieved by using the following special case of a well-known sufficient and necessary condition for two Bessel Gabor systems to be dual of each other, we refer to \cite{Ole4, Jan2, Jan1} for details. We point out that this result is new for $m\geq 3$, but has already been established for $m=1$ (\cite{Lemniel}), and $m=2$ (\cite{Olehon}).

\begin{prop}\label{prop1} Given  $N\geq 2$ and $0<a<N$, suppose that  $g \in V_{N,a}$. For each $m\geq 1$, let  $h$ be a bounded real-function supported on $ [-\frac{2m-1}{2}a,\frac{2m-1}{2}a]$.
 Then the Gabor systems $\cG(g,a, b)$ and $\cG(h,a, b)$ are dual frames for $L^2(\mathbb{R})$ if and only if 
 \begin{equation}\label{eq:dualred}
  \sum_{k=1-m}^{m-1}{g(x-\ell/b+ka)h(x+ka)}=b\delta_{\ell,0},\ |\ell|\leq m-1\ \mbox{a.e.}\ x\in [-\frac{a}{2},\frac{a}{2}]. 
 \end{equation}
\end{prop}

\begin{proof} We only consider the case $m\geq 3$.  Let $ x\in [-\frac{a}{2}, \frac{a}{2}]$. Then 
 $ h(x+ka)\neq0 \iff |k|\leq m-1.$ If in addition, we assume that  $|k|\leq m-1$, then 
$g(x-\ell/b+ka)\neq0\iff |\ell|\leq m-1.$

Since  $g$ and $h$ are bounded with compact support, then the Gabor systems $\cG(g,a, b)$  and $\cG(h,a, b)$ are Bessel sequences for all  $a,b>0$ \cite[Proposition 6.2.2]{Groc2001}. In addition, these systems are dual if and only if 
\begin{equation*}\label{eq:dualcond}
   \sum_{k\in\mathbb{Z}}{g(x-\ell/b+ka)\overline{h(x+ka)}}=b\delta_{\ell,0},\ \mbox{for a.e}\ x\in\left[-\frac{a}{2},\frac{a}{2}\right]. 
\end{equation*}
This last equation   reduces to~\eqref{eq:dualred} by the bounds on $k$ and $\ell$. 
\end{proof}

We can rewrite \eqref{eq:dualred} as a matrix-vector equation.

\begin{equation}
\label{eq-matrix}
 G_m(x)\left(\begin{array}{c}
                            h(x+(1-m)a) \\ \vdots \\ h(x)\\ \vdots \\ h(x+(m-1)a) \\                           
             \end{array}\right)=\left(\begin{array}{r} 0\\ \vdots \\ 0\\ b \\0\\ \vdots \\ 0\end{array}\right),\ \mbox{for a.e}\ x\in\left[-\frac{a}{2},\frac{a}{2}\right].
\end{equation}

 where $G_m(x)$ is the $(2m-1)\times(2m-1)$ matrix-valued function on $\left[-\frac{a}{2},\frac{a}{2}\right]$ defined by
\begin{eqnarray*}
G_m(x)&=& \left[g(x-\frac{\ell}{b}+ka)\right]_{1-m\leq\ell,k\leq m-1}=\\
 \end{eqnarray*}
\begin{equation*}
  \left(\begin{array}{cccccccc}
 g(x+\frac{m-1}{b}+(1-m)a) & \dots & \dots & g(x+\frac{m-1}{b}) &  \dots & \dots & g(x+\frac{m-1}{b}+(m-1)a)\\
 \vdots  & \vdots & \vdots & \vdots & \vdots & \vdots & \vdots \\
 g(x+(1-m)a) & \dots & \dots & g(x) & \dots & \dots & g(x+(m-1)a)\\
\vdots  & \vdots & \vdots & \vdots & \vdots & \vdots & \vdots\\
\vdots  & \vdots & \vdots & \vdots & \vdots & \vdots & \vdots\\
 g(x+\frac{1-m}{b}+(1-m)a) & \dots & \dots  & g(x+\frac{1-m}{b}) & \dots  & \dots  & g(x+\frac{1-m}{b}+(m-1)a)\\ 
 \end{array}\right)
\end{equation*}

The case  $m=1$ corresponds to  the matrix $G_1(x)=g(x)$ which was considered in \cite{Lemniel}. Similarly, the case $m=2$ corresponds to the matrix
$$G_2(x)= \begin{pmatrix}
 g(x+\frac{1}{b}-a) & g(x+\frac{1}{b}) & g(x+\frac{1}{b}+a)\\
 g(x-a) & g(x) & g(x+a)\\
 g(x-\frac{1}{b}-a) & g(x-\frac{1}{b}) & g(x-\frac{1}{b}+a)
\end{pmatrix}$$ considered in \cite{Olehon}. Proposition~\ref{prop1} is illustrated in Figure~\ref{fig:figure2}.

 \begin{figure}[!h]
\includegraphics[scale=0.35]{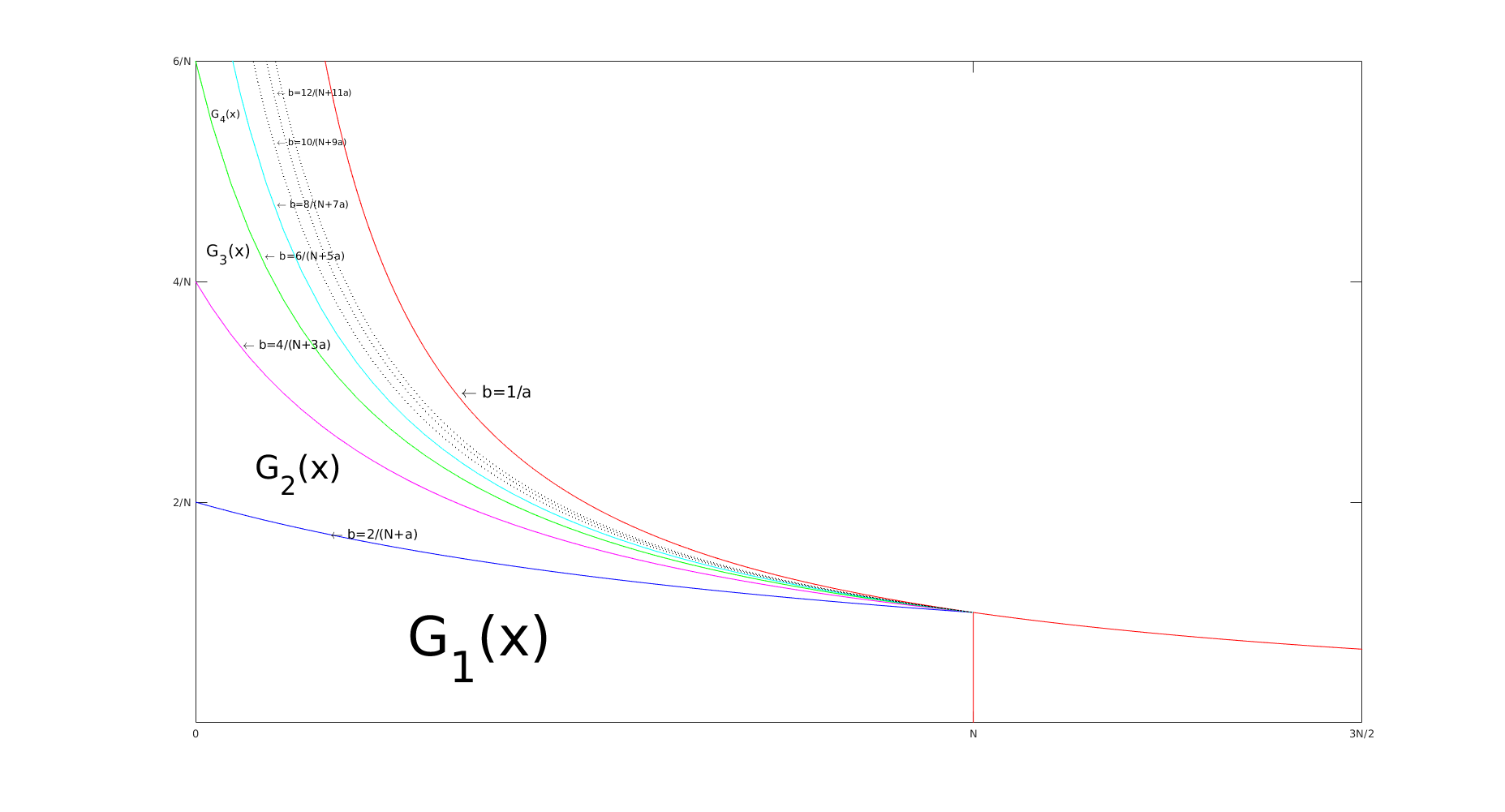}
\caption{Region of interest for the matrix $G_m(x)$ when $m=1,2,3,4, \hdots$ }
 \label{fig:figure2}
\end{figure}

\begin{rem}\label{rem:parityh}
According to Proposition~\ref{prop1}, to prove Theorem~\ref{th:main1} we only need to show, under the assumptions on $g$, that~\eqref{eq:dualred} (or equivalently~\eqref{eq-matrix})  has a unique solution $h$. This is equivalent to proving that the matrix $G_m(x)$ is invertible for $\mbox{a.e.} \ x\in\left[-\frac{a}{2},\frac{a}{2}\right]$. In particular, it is necessary and sufficient to show that $|G_m(x)|\neq 0$ $\mbox{a.e.} \ x\in\left[-\frac{a}{2},\frac{a}{2}\right]$ where $|B|$ denotes the determinant of the square matrix $B$. 
 In addition, since any function 
$g \in V_{N,a}$ is even, it suffices to conduct the  analysis of  the determinant $|G_m(x)|$  on $ [-\frac{a}{2},0]$.

Indeed, for all $x\in [-\frac{a}{2},\frac{a}{2}]$, the symmetry of $g$ implies
\begin{align*}
 |G_m(-x)| &= \mbox{det}\left(g(-x-\frac{\ell}{b}+ka)\right)_{1-m\leq\ell,k\leq m-1}&= \mbox{det}\left(g(x+\frac{\ell}{b}-ka)\right)_{1-m\leq\ell,k\leq m-1}\  \\ 
     &= -\mbox{det}\left(g(x-\frac{\ell}{b}-ka)\right)_{1-m\leq\ell,k\leq m-1}&= \mbox{det}\left(g(x-\frac{\ell}{b}+ka)\right)_{1-m\leq\ell,k\leq m-1}\\
                   &= | G_m(x) |
\end{align*}

In fact, assuming that $|G_m(x)|\neq 0$ on $\left[-\frac{a}{2},\frac{a}{2}\right]$, we can show that the unique solution $h$ to~\eqref{eq:dualred} (or equivalently~\eqref{eq-matrix}) is an even function. Indeed, let $x\in [-a/2,a/2]$ and substitute  $-x\in [-a/2, a/2]$ in~\eqref{eq:dualred}. We have

\begin{align*}
 \sum_{k=1-m}^{m-1}g(-x-\ell/b+ka)\overline{h(-x+ka)}&=  b\delta_{\ell,0}\\
 &= \sum_{k=1-m}^{m-1} g(x+\ell/b-ka)\overline{h(-x+ka)}\\
 &= \sum_{k'=1-m}^{m-1} g(x+\ell/b+k'a)\overline{h(-x-k'a)}\\
 &=\sum_{k'=1-m}^{m-1} g(x-(-\ell)/b+k'a)\overline{\tilde{h}(x+k'a)}\\
 &=b\delta_{-\ell,0}
 \end{align*}
where $\tilde{h}(x)=h(-x)$. Hence, the uniqueness of the solution of ~\eqref{eq:dualred}, implies that for each $x\in [-a/2, a/2]$ and $k=1-m, \hdots, 0, \hdots, m-1$, $$\tilde{h}(x+k'a)=h(-x-ka)=h(x+ka).$$ Consequently, we only need to define the function $h$ on half of the interval $[-\tfrac{2m-1}{2}a, \tfrac{2m-1}{2}a]$.
\end{rem}

The next result specifies some of the entries of the  matrix $G_m(x)$. 

\begin{lemma}\label{lem:diagonal} Given  $N\geq 2$ and $0<a<N$, suppose that  $g \in V_{N,a}$. Assume that $m\geq 3$,  and let $(a,b)\in T_m$.  If $x\in [-\frac{a}{2},0]$, then the following hold.
\begin{enumerate}
 \item[(a)]  $g(x+\frac{k}{b}-ka)>0,$  for all $|k|\leq m-1$.
 \item[(b)] $g(x-\frac{k}{b}+(k-1)a)=0,$  for all $k \in \left\{1, \hdots, m-1\right\}$.
 \item[(c)] If $k\in\left\{1, \hdots, m-1\right\}$, then $g(x-\frac{k}{b}+\ell a)=0,$  for all $\ell \in \left\{1-m,...,k-2\right\}$. 
 \item[(d)] $g(x+\frac{k}{b}+(2-k)a)=0,$  for all $k \in \left\{3-m, \hdots, m-1\right\}$.
 \item[(e)] If  $k \in \left\{3-m, \hdots, m-1\right\},$ then $g(x+\frac{k}{b}+\ell a)=0,$  for all $\ell \in \left\{3-k, \hdots, m-1\right\}$, \mbox{$ \ell\neq m$}.
\end{enumerate}
\end{lemma}

\begin{proof}

\begin{enumerate}
\item[(a)] We first show the result for $k=1-m, $ and $k=m-1$.
For  $x\in [-a/2, 0]$ we see that 
$-\frac{a}{2}+\frac{m-1}{b}+(1-m)a\leq x+\frac{m-1}{b}+(1-m)a\leq\frac{m-1}{b}+(1-m)a$. Next, using the following inequalities, 
\begin{equation}\label{eq:ineq1}
\frac{2(m-1)}{N+(2m-3)a}<b\leq\frac{2m}{N+(2m-1)a}
\end{equation}
we get

\begin{align*}
 x+\frac{m-1}{b}+(1-m)a&\leq\frac{m-1}{b}+(1-m)a\\
                        &<(m-1)\frac{N+(2m-3)a}{2(m-1)}+(1-m)a \\ 
                        &=\frac{N}{2}-\frac{a}{2}<\frac{N}{2}
\end{align*} where we have also used the fact that $a<N/2$. 

\noindent
On the other hand,  using \eqref{eq:ineq1} and 
\begin{equation}\label{eq:ineq2}
\frac{N(m-2)}{2m-3}<a<\frac{N}{2}
\end{equation}
we get 

\begin{align*}
 x+\frac{m-1}{b}+(1-m)a &\geq-\frac{a}{2}+\frac{m-1}{b}+(1-m)a\\
                        &\geq-\frac{a}{2}+(m-1)\frac{N+(2m-1)a}{2m}+(1-m)a \\
                        &=\frac{(m-1)N}{2m}-\frac{(2m-1)}{2m}a\\
                        &> \frac{N}{4m}>-\frac{N}{2}.     
\end{align*} 
Since $g$ is strictly positive on $\left]-\frac{N}{2},\frac{N}{2}\right[$, then $$g(x+\frac{m-1}{b}+(1-m))>0, \forall x\in\left[-\frac{a}{2},0\right]$$

\noindent
A similar argument leads to the fact that 
 $g(x+\frac{1-m}{b}+(m-1))>0$.
 
\noindent
Now, let $ |k|\leq m-2$. Then

\begin{align*}
  x+\frac{k}{b}-ka &= x+\frac{m-1}{b}+(1-m)a+(k-(m-1))\left(\frac{1}{b}-a\right)\\    
                            &\leq      x+\frac{m-1}{b}+(1-m)a < \tfrac{N}{2}       
\end{align*}
since, $k\leq m-2,$ and $x+\frac{m-1}{b}+(1-m)a<\frac{N}{2}$ was established earlier. 

\noindent
Similarly, one shows that $-\frac{N}{2}<  x+\frac{k}{b}-ka.$

\noindent
Thus for all $|k|\leq m-2$, $ -\frac{N}{2}<x+\frac{k}{b}-ka<\frac{N}{2}$, which  concludes the proof.

\item[(b)] 
We start with the case $k=1$ and show that $g\left(x-\frac{1}{b}\right)=0$ for all \mbox{$x\in\left[-\frac{a}{2},0\right]$}. But from the definition of $T_m$ we see that for each $x\in [-a/2, 0]$, then \mbox{$x-1/b<-N/2$}, which gives the result. 

\noindent
Next, for all $k \in \left\{2,...,m-1\right\}$, we have 
$$ x-\frac{k}{b}+(k-1)a = x-\frac{1}{b}+(k-1)\left(a-\frac{1}{b}\right)<-\frac{N}{2}$$
 which follows from the case $k=1$ and the fact that $(k-1)\left(a-\frac{1}{b}\right)\leq0$. The result now follows from the support condition of $g$.

\item[(c)]   This is proved exactly as case (b).  Indeed,  for $k \in \left\{1,...,m-1\right\}$,  let \mbox{$\ell \in \left\{1-m,...,k-2\right\}$}, then 
$$x-\frac{k}{b}+\ell a =x-\frac{k}{b}+(k-1)a+\left(\ell-(k-1)\right)a.$$
But, $\ell-(k-1)\leq 0$ and  $x-\frac{k}{b}+(k-1)a<-\tfrac{N}{2}$ as shown above.

\item[(d)]  Let us prove $g\left(x+\frac{3-m}{b}+(m-1)a\right)=0, $ for all $x\in\left[-\frac{a}{2},0\right]$. 

Let $x\in [-a/2, 0]$. From the definition of $T_m$ we have
\begin{align*}
x+\frac{3-m}{b}+(m-1)a &\geq-\frac{m-3}{b}+\frac{2m-3}{2}a\\
&>-\frac{N+(2m-3)a}{2(m-1)}(m-3)+\frac{(2m-3)a}{2}\\
&=-\frac{(m-3)N}{2(m-1)}+\frac{2(2m-3)a}{2(m-1)}>N/2.
\end{align*}
It follows  that $g\left(x+\frac{3-m}{b}+(m-1)a\right)=0$.

\noindent
Next, for all $k \in \left\{2-m, \hdots, m-1\right\}$, it follows that 
$$ x+\frac{k}{b}+(2-k)a =  x+\frac{3-m}{b}+(m-1)a+(k-(3-m))\left(\frac{1}{b}-a\right)> \frac{N}{2}.$$

\noindent
The results immediately follow from the case $k=3-m$ and the fact that  $k-(3-m)\geq 0$.

\item[(e)]  This follows from (d). Indeed, let  $k\in\left\{m-1,...,3-m\right\}$ and let $\ell \in \left\{3-k,...,m-1\right\}$.
\begin{align*}
 x+\frac{k}{b}+\ell a &= x+\frac{k}{b}+(2-k)a +\left(\ell-(2-k)\right)a\\ 
  &> x+\frac{k}{b}+(2-k)a.
\end{align*}
The result again follows from the fact that $ \ell-(2-k)\geq 0$ and 
\mbox{$ x+\frac{k}{b}+(2-k)a>\frac{N}{2}$} as seen from case (d). 
\end{enumerate}
\end{proof}

\begin{rem}\label{rem:gmblock}
Lemma~\ref{lem:diagonal} allows us to write  $G_m(x)$ as a block matrix: for $x\in [-a/2, 0]$,
$$G_m(x)=\begin{bmatrix}A_m(x) & B_{m}(x)\\ 0 & C_m(x)\end{bmatrix}$$ where $B_m(x)$ is an $m\times (m-1)$ matrix, $C_m(x)$ is an $(m-1)\times (m-1)$ matrix, and $0$ is a $(m-1)\times m$ zero matrix. 
In fact, $A_m(x)$ can be viewed as  the $m\times m$ submatrix of $G_m(x)$ obtained from deleting the bottom $m-1$ rows and the rightmost  $m-1$ columns of $G_m(x)$. In particular, $A_{m}(x)$ is given by 
\begin{eqnarray*}
A_m(x)&=& \left[g(x-\frac{\ell}{b}+ka)\right]_{1-m\leq\ell,k\leq 0}
\end{eqnarray*}

$$\begin{pmatrix}
 g(x+\frac{m-1}{b}+(1-m)a) & g(x+\frac{m-1}{b}+(2-m)a) & \hdots &\hdots  \\
 g(x+\frac{m-2}{b}+(1-m)a) & g(x+\frac{m-2}{b}+(2-m)a) & \hdots \\
 \cdot & g(x+\frac{m-3}{b}+(2-m)a) & \hdots \\
 \cdot &\hdots & g(x+\frac{3}{b}-2a) & \hdots \\
 \cdot  & \hdots  & g(x+\frac{2}{b}-2a) & g(x+\frac{2}{b}-a) & \cdot\\
\cdot  & \hdots &  g(x+\frac{1}{b}-2a) & g(x+\frac{1}{b}-a) & g(x+\frac{1}{b})\\
\ 0  & 0& \hdots& g(x-a) & g(x)\end{pmatrix}$$

In addition, $A_m(x)$ is a tridiagonal matrix, which can be written as  

\begin{eqnarray*}
\hspace{-0.3cm}
\scriptsize{
\begin{pmatrix}
 g(x+\frac{m-1}{b}+(1-m)a) 	& g(x+\frac{m-1}{b}+(2-m)a) 	& 0 							&\hdots 				 	&\hdots				& 0 \\
 g(x+\frac{m-2}{b}+(1-m)a) 	& g(x+\frac{m-2}{b}+(2-m)a) 	& g(x+\frac{m-2}{b}+(3-m)a)  & \hdots 			&\hdots 			& 0 \\
 0 							& g(x+\frac{m-3}{b}+(2-m)a) 	& g(x+\frac{m-3}{b}+(3-m)a) 	& 0				&\hdots 				& 0 \\
 0 							& 0 							& g(x+\frac{m-4}{b}+(3-m)a) 	& g(x+\frac{3}{b}-2a) 		& 0		& 0 \\
 0 							& 0 							& 0  				& g(x+\frac{2}{b}-2a) 	& g(x+\frac{2}{b}-a)	& 0 \\
 0 							& 0 							& \hdots 	& g(x+\frac{1}{b}-2a) 	& g(x+\frac{1}{b}-a)	& g(x+\frac{1}{b})\\
 0 							& 0 							& 0 					& 0  					& g(x-a) 			& g(x)
 \end{pmatrix}}
\end{eqnarray*}

Furthermore, all the entries of $B_m(x)$ are $0$ except its $(m, 1)$ entry which is $g(x+a)$, and $C_m(x)$ is given by 

\begin{eqnarray*}
\scriptsize{ 
\begin{pmatrix}g(x-1/b+a) & g(x-1/b+2a)& 0& \hdots \\
0& g(x-2/b+2a)& g(x-2/b+3a)& 0& \hdots \\
\vdots & \vdots & \vdots & \vdots & \vdots  \\
0& 0& \hdots & g(x-(m-2)/b+(m-2)a) &g(x-(m-2)/b+(m-1)a)\\
0& 0& \hdots &0& g(x-(m-1)/b+(m-1)a)\end{pmatrix}
}
\end{eqnarray*}

\end{rem}

The following trivial  inequalities can be derived  from  the definition of $T_m$ and  will be used to analyze the entries of the matrix $A_m(x)$.

\begin{lemma}\label{lem:inequalities} Given  $N\geq 2$ and $0<a<N$, suppose that  $m\geq 2$. If $(a,b)\in T_m$ and  $k\in \left\{1,...,m-1\right\}$, then the following hold.
 \begin{enumerate}
  \item[(a)]  If $x\in [-\frac{a}{2}, 0]$, then $-\frac{N}{2}<x+\frac{k}{b}-ka<\frac{N}{2}$.
  \item[(b)]  If $x \in [-\frac{a}{2}, \frac{N}{2}-\frac{k}{b}+(k-1)a)$, then $ 0<x+\frac{k}{b}-(k-1)a<\frac{N}{2}$; and if   $ x\in [\frac{N}{2}-\frac{k}{b}+(k-1)a, 0],$ then $ 
  \frac{N}{2}\leq x+\frac{k}{b}-(k-1)a<\frac{N}{2}+\frac{Nk}{4(m-1)}$.
  \item[(c)]  If $ x\in [-\frac{a}{2}, ka+\frac{1-k}{b}-\frac{N}{2}],$ then  $-\frac{N}{2}-\frac{(m-k+1)N}{4m}<x+\frac{k-1}{b}-ka\leq-\frac{N}{2}$;  and if $ x\in (ka+\frac{1-k}{b}-\frac{N}{2}, 0],$ then  
  $-\frac{N}{2}\leq x+\frac{k-1}{b}-ka<0$.  
  \item[(d)]  $\frac{N}{2}-\frac{k+1}{b}+ka<\frac{N}{2}-\frac{k}{b}+(k-1)a$, and $(1+k)a-\frac{k}{b}-\frac{N}{2}<ka+\frac{1-k}{b}-\frac{N}{2}$
\end{enumerate}
\end{lemma}

Let $k\in \left\{1,...,m-1\right\}$ and consider the $2\times2$ submatrix $A_{k,k-1}(x)$ of $A_m(x)$ defined by
$$
A_{k,k-1}(x)= \begin{pmatrix} 
 g(x+\frac{k}{b}-ka) & g(x+\frac{k}{b}-(k-1)a)\\ 
 g(x+\frac{k-1}{b}-ka) & g(x+\frac{k-1}{b}-(k-1)a)
\end{pmatrix}$$

The following lemma shows that the matrix $A_{k,k-1}(x)$ is invertible for all \mbox{$k\in \left\{1,...,m-1\right\}$}.

\begin{lemma}\label{lem:submatrices}
Given  $N\geq 2$ and $0<a<N$, suppose that  $g \in V_{N,a}$. Assume that $m\geq 2$,  and let $(a,b)\in T_m$.  If $x\in [-\frac{a}{2},0]$, then for all $k\in \left\{1,...,m-1\right\}$
 \begin{equation}
  \left|A_{k,k-1}(x)\right|>0.
 \end{equation}
\end{lemma}

\begin{proof}

 First, we prove that $g(x+\frac{k}{b}-ka)> g(x+\frac{k}{b}-(k-1)a)$.
 From (a) and (b) of Lemma~\ref{lem:inequalities} we know that $-\frac{N}{2}<x+\frac{k}{b}-ka<\frac{N}{2}$ and $x+\frac{k}{b}-(k-1)a>0$, therefore we have two different cases.
 
 $\bullet$ If $x+\frac{k}{b}-ka>0$, then the monotonicity of $g$ on $[0, \frac{N}{2}]$  implies that for 
 
\mbox{$0<x+\frac{k}{b}-ka<x+\frac{k}{b}-(k-1)a,$} we have  $g(x+\frac{k}{b}-ka)> g(x+\frac{k}{b}-(k-1)a).$

$\bullet$ If $x+\frac{k}{b}-ka\leq0$, we have $-x-\frac{k}{b}+ka\geq0$ and 
$$-x-\frac{k}{b}+ka-(x+\frac{k}{b}-(k-1)a)=-2x-\frac{2k}{b}+2ka-a<0.$$ 
Consequently, $g(-x-\frac{k}{b}+ka)>g(x+\frac{k}{b}-(k-1)a)$, and by the symmetry of $g$ we get $g(x+\frac{k}{b}-ka)>g(x+\frac{k}{b}-(k-1)a)$. 
 
Next,  we show that $g(x+\frac{k-1}{b}-ka)<g(x+\frac{k-1}{b}-(k-1)a)$.
From (a) and (c) of Lemma~\ref{lem:inequalities}, we know  also that $-\frac{N}{2}<x+\frac{k-1}{b}-(k-1)a<\frac{N}{2}$ and $x+\frac{k-1}{b}-ka<0$, therefore we can consider the following two cases.

$\bullet$ If $x+\frac{k-1}{b}-(k-1)a\leq0$, we have 
$$
x+\frac{k-1}{b}-ka<x+\frac{k-1}{b}-(k-1)a
$$
then 
$$
g(x+\frac{k-1}{b}-ka)<g(x+\frac{k-1}{b}-(k-1)a)
$$ 
because $g$ is strictly increasing on $ [-\frac{N}{2}, 0]$.

$\bullet$ If $x+\frac{k-1}{b}-(k-1)a>0$, we have 
$$
-x-\frac{k-1}{b}+ka-(x+\frac{k-1}{b}-(k-1)a)=-2x-\frac{2(k-1)}{b}+(2k-1)a>0.
$$  Using this as well as the monotonicity and the symmetry of $g$ we see  that 

$$
g(-x-\frac{k-1}{b}+ka)<g(x+\frac{k-1}{b}-(k-1)a),
$$ 
or, equivalently, 
$$
g(x+\frac{k-1}{b}-ka)<g(x+\frac{k-1}{b}-(k-1)a).
$$ 
All together, we have proved that
$$ \left\{
\begin{array}{cc}
g(x+\frac{k-1}{b}-ka)<g(x+\frac{k-1}{b}-(k-1)a)\\
g(x+\frac{k}{b}-ka)> g(x+\frac{k}{b}-(k-1)a),
\end{array}
\right.$$ 
which gives 
$$
g(x+\frac{k}{b}-ka)g(x+\frac{k-1}{b}-(k-1)a)>g(x+\frac{k-1}{b}-ka)g(x+\frac{k}{b}-(k-1)a).
$$
Consequently, for each $x\in [-a/2, 0]$, $|A_{k,k-1}(x)|>0,$ for all $k\in \left\{1,...,m-1\right\}.$

\end{proof}

For the next result, we recall that from the definition of $T_m$, it is easy to see that for $(a,b)\in T_m$, 

\begin{equation*}
 -\frac{(m+1)N}{2m(2m-3)} < a-N+\frac{1}{b} < \frac{N}{4(m-1)}
\end{equation*}

\begin{lemma}\label{lem:3cases}
Given  $N\geq 2$ and $0<a<N$, let  $m\geq 2$. If  $(a,b)\in T_m$ and  \mbox{$k\in \left\{1,...,m-1\right\}$}, then the following statements hold.

\begin{enumerate}
\item[(a)]  If $\frac{N}{2}-\frac{1}{b}=a-\frac{N}{2}$, then 
    
    $$
    \frac{N}{2}-\frac{k}{b}+(k-1)a=ka+\frac{1-k}{b}-\frac{N}{2}<\frac{N}{2}-\frac{k-1}{b}+(k-2)a.
  $$
 \item[(b)] If $\frac{N}{2}-\frac{1}{b}< a-\frac{N}{2}$, then 
   $$ \frac{N}{2}-\frac{k}{b}+(k-1)a< ka+\frac{1-k}{b}-\frac{N}{2}<\frac{N}{2}-\frac{k-1}{b}+(k-2)a.$$
  
\item[(c)]   If $a-\frac{N}{2}<\frac{N}{2}-\frac{1}{b}$, then 
 $$  ka+\frac{1-k}{b}-\frac{N}{2}<\frac{N}{2}-\frac{k}{b}+(k-1)a<(k-1)a+\frac{2-k}{b}-\frac{N}{2}.$$
  \end{enumerate}
\end{lemma}

\begin{proof}  The result easily follows from the fact that  for all $k \in \left\{1,...,m-1\right\}$, we have 
$$\frac{N}{2}-\frac{k}{b}+(k-1)a=k(a-\frac{1}{b})+(\frac{N}{2}-a),$$  $$ka+\frac{1-k}{b}-\frac{N}{2}=k(a-\frac{1}{b})+(\frac{1}{b}-\frac{N}{2}),$$

$$\left(ka+\frac{1-k}{b}-\frac{N}{2}\right)-\left(\frac{N}{2}-\frac{k-1}{b}+(k-2)a\right)=2(a-\frac{N}{2})<0,$$
and
$$\left(\frac{N}{2}-\frac{k}{b}+(k-1)a\right)-\left((k-1)a+\frac{2-k}{b}-\frac{N}{2}\right)=2(\frac{N}{2}-\frac{1}{b})<0.$$ 
\end{proof}

We can now give an explicit expression for the determinant of $A_m(x)$ when $m\geq3$,  $x\in [-\frac{a}{2}, 0]$, and under the hypotheses of Theorem \ref{th:main1}.  The different cases considered in proving this result are illustrated for the cases $m=2$ and $m=3$ in Figure~\ref{fig:figure3}.

\begin{lemma}\label{lem:subdeter} Given  $N\geq 2$ and $0<a<N$, suppose that  $g \in V_{N,a}$. Assume that $m\geq3$, and  let $(a,b)\in T_m$. Then the following statements hold.
  \begin{enumerate}
 \item[(a)] If $\frac{N}{2}-\frac{1}{b}\leq a-\frac{N}{2}$, then for all $x\in [-a/2, 0]$
\begin{equation*}
 |A_m(x)|=\prod_{k=0}^{m-1}{g(x+\frac{k}{b}-ka)}.
\end{equation*}

\item[(b)]  If $a-\frac{N}{2}<\frac{N}{2}-\frac{1}{b}$, then for each $x\in S_m$, 
\begin{equation*}
 |A_m(x)|=\prod_{k=0}^{m-1}{g(x+\frac{k}{b}-ka)},
\end{equation*} where

 \begin{align*}
 S_m&=\left[-\frac{a}{2}, (m-1)a+\frac{2-m}{b}-\frac{N}{2}\right] \cup \left[\frac{N}{2}-\frac{1}{b}, 0\right]\\
 & \bigcup_{k=1-m}^{-2} \left[\frac{N}{2}+\frac{k}{b}-(k+1)a, -(k+1)a+\frac{k+2}{b}-\frac{N}{2}\right]. 
 \end{align*}
 
In addition, for  each 
$$
x\in Z_m=\displaystyle{\left[-\frac{a}{2} , 0\right]}\setminus S_m=\bigcup_{\ell=1-m}^{-1}{\bigg(-\ell a+\frac{\ell+1}{b}-\frac{N}{2}, \frac{N}{2}+\frac{\ell}{b}-(\ell+1)a\bigg)},
$$ 
let $\ell \in \{1-m, \hdots, -1\}$ be the unique integer such that 
$$
x\in \bigg(-\ell a+\frac{\ell+1}{b}-\frac{N}{2}, \frac{N}{2}+\frac{\ell}{b}-(\ell+1)a\bigg)
$$
then 
\begin{equation*}
 |A_m(x)|= |A_{-\ell, -\ell-1}(x)| \times \prod_{  \begin{array}{c}\scriptstyle k=0 \\ \scriptstyle k\neq -\ell, -\ell-1\end{array}}^{m-1} g(x+\frac{k}{b}-ka).
\end{equation*} 
\end{enumerate}
\end{lemma}

\begin{proof}
 We prove the result by induction on $m$. 
 
 We recall that the cases $m=1$ and $m=2$ have already been settled. Indeed, $|A_1(x)|=g(x)>0$ and $|A_2(x)|=|A_{1,0}(x)|>0$. 

 For $m=3$, the matrix $A_3(x)$ is given by 
$$
A_3(x)=\begin{pmatrix}
  g(x+\frac{2}{b}-2a) & g(x+\frac{2}{b}-a) & 0\\
  g(x+\frac{1}{b}-2a) & g(x+\frac{1}{b}-a) & g(x+\frac{1}{b})\\
  0 & g(x-a) & g(x)\end{pmatrix} $$
Let $x \in [-\frac{a}{2}, 0].$ From Lemma \ref{lem:inequalities}, we know that $g(x+\frac{2}{b}-2a)>0$. Furthermore, 
\begin{eqnarray*}
 g(x+\frac{2}{b}-a)=
 \begin{dcases}
  g(x+\frac{2}{b}-a) >0	& \mbox{if}\ x \in \left[-\frac{a}{2}, \frac{N}{2}-\frac{2}{b}+a\right)\\
\qquad \qquad  0 		& \mbox{if}\ x \in \left[\frac{N}{2}-\frac{2}{b}+a, 0\right]
\end{dcases}
\end{eqnarray*}
and
\begin{eqnarray*}
g(x+\frac{1}{b}-2a)=
\begin{dcases}
\qquad  \qquad   0 			& \mbox{if}\ x \in \left[-\frac{a}{2}, 2a-\frac{1}{b}-\frac{N}{2} \right]\\
   g(x+\frac{1}{b}-2a)>0		& \mbox{if}\ x \in \left(2a-\frac{1}{b}-\frac{N}{2}, 0 \right] 
\end{dcases}
\end{eqnarray*}

\begin{enumerate}
\item[(a)]  If $\frac{N}{2}-\frac{1}{b}\leq a-\frac{N}{2}$, then $-\frac{a}{2}<\frac{N}{2}-\frac{2}{b}+a\leq 2a-\frac{1}{b}-\frac{N}{2}<\frac{N}{2}-\frac{1}{b}\leq a-\frac{N}{2}<0$. We can consider the following  cases.

\begin{enumerate}
\item[(a-1)] If $x\in \displaystyle{\left[-\frac{a}{2}, 2a-\frac{1}{b}-\frac{N}{2}\right]}$, then $g(x-a)=g(x+1/b-2a)=0$. $A_3(x)$ is  an upper triangular matrix and its determinant is the product of the diagonal entries.
\item[(a-2)]
If $x\in \displaystyle{ \left[2a-\frac{1}{b}- \frac{N}{2}, \frac{N}{2}-\frac{1}{b}\right]}$, then $g(x-a)=g(x+2/b-a)=0$. Computing the determinant of $A_3(x)$ along the first row gives the result. 
\item[(a-3)]
If $x\in \displaystyle{ \left[\frac{N}{2}- \frac{1}{b}, 0 \right]}$, then $g(x+2/b-a)=g(x+1/b)=0$. $A_3(x)$ is thus a lower triangular matrix and its determinant is the product of the diagonal entries.
\end{enumerate}

\noindent
This establishes part (a) for the base case $m=3$. Suppose that (a) holds for $m\geq 3$ and let us prove that it holds for $m+1$. Using the Laplace expansion  by minors along the first row, we have
\begin{equation*}
 \left|A_{m+1}(x)\right|=g(x+\frac{m}{b}-ma)\left|A_m(x)\right|-g(x+\frac{m}{b}+(1-m)a)g(x+\frac{m-1}{b}-ma)\left|A_{m-1}(x)\right|.
\end{equation*}

\noindent
From Lemma \ref{lem:inequalities}, we know that for all $x \in [-\frac{a}{2}, 0]$ we have $g(x+\frac{m}{b}-ma)>0.$  Furthermore, 
\begin{eqnarray*}
g(x+\frac{m}{b}+(1-m)a)=
\begin{dcases}
g(x+\frac{m}{b}+(1-m)a) > 0 	& \mbox{if} \ x  \in \left[-\frac{a}{2}, \frac{N}{2}-\frac{m}{b}+(m-1)a\right)\\
\qquad \qquad  0  			& \mbox{if} \ x  \in \left[\frac{N}{2}-\frac{m}{b}+(m-1)a, 0\right]
\end{dcases}
\end{eqnarray*}
and
\begin{eqnarray*}
  g(x+\frac{m-1}{b}-ma)=
\begin{dcases}
 \qquad \qquad  0 			&	\mbox{if}\ x \in \left[-\frac{a}{2}, ma-\frac{m-1}{b}-\frac{N}{2}\right]\\
g(x+\frac{1}{b}-2a)>0		& 	\mbox{if}\ x \in \left(ma-\frac{m-1}{b}-\frac{N}{2}, 0\right] 
\end{dcases}
\end{eqnarray*}
Using  Lemma \ref{lem:3cases} and the definition of $T_m$, we have 

 \begin{align*}
 -\frac{a}{2}&<\frac{N}{2}-\frac{m}{b}+(m-1)a\leq ma-\frac{m-1}{b}-\frac{N}{2}<\\
 &< \tfrac{N}{2}-\tfrac{m-1}{b}+(m-2)a\leq  (m-1)a-\frac{m-2}{b}-\frac{N}{2}<\hdots\leq \\
 & \leq 2a-\frac{1}{b}-\frac{N}{2}<\frac{N}{2}-\frac{1}{b}\leq a-\frac{N}{2}<0.
 \end{align*}
Thus for all 
$x\in [-\frac{a}{2}, 0]$,  $g(x+\frac{m}{b}+(1-m)a) g(x+\frac{m-1}{b}-ma)=0$, leading to 
$$|A_{m+1}(x)| = g(x+\frac{m}{b}-ma)\,|A_m(x)|=g(x+\frac{m}{b}-ma)\,\prod_{k=0}^{m-1}{g(x+\frac{k}{b}-ka)}= \prod_{k=0}^{m}{g(x+\frac{k}{b}-ka)}.$$ 
This establishes part (a).

\item[(b)] We first consider the base case $m=3$. Similarly to case (a),  if $a-\frac{N}{2}<\frac{N}{2}-\frac{1}{b}$, then $-\frac{a}{2}<2a-\frac{1}{b}-\frac{N}{2}<\frac{N}{2}-\frac{2}{b}+a<a-\frac{N}{2}<\frac{N}{2}-\frac{1}{b}<0$. Therefore we can consider the following  cases.

\begin{enumerate}
\item[(b-1)] If $x\in [-\frac{a}{2}, 2a-\frac{1}{b}-\frac{N}{2} ]\cup [\frac{N}{2}-\frac{2}{b}+a, a-\frac{N}{2}]\cup [\frac{N}{2}-\frac{1}{b}, 0]$,  $A_3(x)$ is either a triangular matrix, or its second column has only one nonzero entry, which is the diagonal entry. In any of these cases the determinant of $A_3(x)$ is the product of the diagonal entries.

\item[(b-2)] If  $x\in (2a-\frac{1}{b}-\frac{N}{2}, \frac{N}{2}-\frac{2}{b}+a)$  we have   $g(x-a)=0$ leading to 
\begin{equation*}
 |A_3(x)|=g(x) |A_{2,1}(x)|.
\end{equation*}

\item[(b-3)] If   $x\in (a-\frac{N}{2}, \frac{N}{2}-\frac{1}{b})$, we have $g(x+\frac{2}{b}-a)=0$ leading to 
\begin{equation*}
 |A_3(x)|=g(x+\frac{2}{b}-2a) |A_{1,0}(x)|. 
\end{equation*}

\end{enumerate}

Consequently, for each $x\in \displaystyle{\bigcup_{\ell =-2}^{-1} }\left(-\ell 
 a+\frac{\ell+1}{b}-\frac{N}{2}, \frac{N}{2}+\frac{\ell}{b}-(\ell+1)a \right), $ there exists a unique $\ell \in \{-2, -1 \}$ with $x\in \displaystyle{\left(-\ell 
 a+\frac{\ell+1}{b}-\frac{N}{2};\frac{N}{2}+\frac{\ell}{b}-(\ell+1)a\right)}$ and 

\begin{equation*}
 |A_3(x)|=  |A_{-\ell,-\ell-1}(x)|  \times \prod_{\begin{array}{c}\scriptstyle k=0 \\ \scriptstyle k\neq-\ell,-\ell-1\end{array}}^{2}{g(x+\frac{k}{b}-ka)}.
\end{equation*}

Now, suppose that part (b) holds  $m\geq3$ and  let us prove that it holds for  $m+1$.  Proceeding as above we have.

If $a-\frac{N}{2}<\frac{N}{2}-\frac{1}{b}$, then 
$$-\frac{a}{2}<ma-\frac{m-1}{b}-\frac{N}{2}<\frac{N}{2}-\frac{m}{b}+(m-1)a<(m-1)a-\frac{m-2}{b}-\frac{N}{2}$$
and 
$$(m-1)a-\frac{m-2}{b}-\frac{N}{2}<\frac{N}{2}-\frac{m-1}{b}+(m-2)a<\hdots<\frac{N}{2}-\frac{2}{b}+a <a-\frac{N}{2}<\frac{N}{2}-\frac{1}{b}<0$$
Therefore, for all $x \in S_{m+1}$,
we have  $g(x+\frac{m}{b}+(1-m)a)\, g(x+\frac{m-1}{b}-ma)=0$ leading to 
\begin{eqnarray*}
 |A_{m+1}(x)| &=& g(x+\frac{m}{b}-ma)|A_m(x)|
\end{eqnarray*}

Now let $x\in Z_{m+1}={\displaystyle \cup_{\ell=1}^{m}} I_\ell$ where 
$$
I_\ell= \left(\ell a-\tfrac{\ell-1}{b}-\tfrac{N}{2}, \tfrac{N}{2}-\tfrac{\ell}{b}+(\ell-1)a\right).
$$ Then for some $\ell\in \{1, \hdots, m\}$, $x\in I_\ell.$  Suppose that $\ell\leq m-1$ and let $x\in I_\ell$. The induction argument gives $$ |A_m(x)|= |A_{-\ell, -\ell-1}(x)|\times  \prod_{  \begin{array}{c}\scriptstyle k=0 \\ \scriptstyle k\neq -\ell, -\ell-1\end{array}}^{m-1} g(x+\frac{k}{b}-ka).$$
Consequently,

\begin{align*}
|A_{m+1}(x)|&= g(x+\frac{m}{b}-ma)|A_m(x)|\\
&= g(x+\frac{m}{b}-ma)  |A_{-\ell, -\ell-1}(x)| \times \prod_{  \begin{array}{c}\scriptstyle k=0 \\ \scriptstyle k\neq -\ell, -\ell-1\end{array}}^{m-1} g(x+\frac{k}{b}-ka)\\
&=|A_{-\ell, -\ell-1}(x)|\times  \prod_{  \begin{array}{c}\scriptstyle k=0 \\ \scriptstyle k\neq -\ell, -\ell-1\end{array}}^{m} g(x+\frac{k}{b}-ka).
\end{align*}

Finally, suppose that  $\ell=m$ and let $x\in I_m$. Then 
$g(x+\frac{m}{b}+(1-m)a)\, g(x+\frac{m-1}{b}-ma)>0.$ 
Note  that 
$$
I_m \subset \left[-\frac{a}{2}, (m-1)a-\frac{m-2}{b}-\frac{N}{2}\right] \subset \left[-\frac{a}{2}, (m-2)a-\frac{m-3}{b}-\frac{N}{2}\right]
$$ 
and that 
$$
|A_{m+1}(x)|=g(x+\frac{m}{b}-ma)\, |A_m(x)|-g(x+\frac{m}{b}+(1-m)a)g(x+\frac{m-1}{b}-ma)\, |A_{m-1}(x)|.
$$
Hence by the induction assumption, we have respectively 
$$
|A_m(x)|=\prod_{k=0}^{m-1}{g(x+\frac{k}{b}-ka)}\ \mbox{and}\ |A_{m-1}(x)|=\prod_{k=0}^{m-2}{g(x+\frac{k}{b}-ka)}.
$$

\noindent
Consequently, 
\begin{eqnarray*}
\left|A_{m+1}(x)\right|  &=& g(x+\frac{m}{b}-ma)\left|A_m(x)\right|-g(x+\frac{m}{b}+(1-m)a)g(x+\frac{m-1}{b}-ma)\left|A_{m-1}(x)\right|\\
                         &=& g(x+\frac{m}{b}-ma)\prod_{k=0}^{m-1}{g(x+\frac{k}{b}-ka)}-\\
                         && g(x+\frac{m}{b}+(1-m)a)g(x+\frac{m-1}{b}-ma) \prod_{k=0}^{m-2}{g(x+\frac{k}{b}-ka)}\\
                         &=& \prod_{k=0}^{m-2}{g(x+\frac{k}{b}-ka)}\left[g(x+\frac{m}{b}-ma)g(x+\frac{m-1}{b}-(m-1)a)- \right. 
                         \\
                         && \left. g(x+\frac{m}{b}+(1-m)a)g(x+\frac{m-1}{b}-ma)\right]\\
                         &=& \prod_{k=0}^{m-2}{g(x+\frac{k}{b}-ka)}\left|A_{m,m-1}(x)\right|\\
                         &=& \prod_{\begin{array}{c}\scriptstyle k=0 \\ \scriptstyle k\neq m,m-1\end{array}}^{m}{g(x+\frac{k}{b}-ka)\left|A_{m,m-1}(x)\right|}
\end{eqnarray*}
which concludes the proof. 
\end{enumerate}

\end{proof}

\begin{figure}
\includegraphics[width=1.5\textwidth,angle=-90]{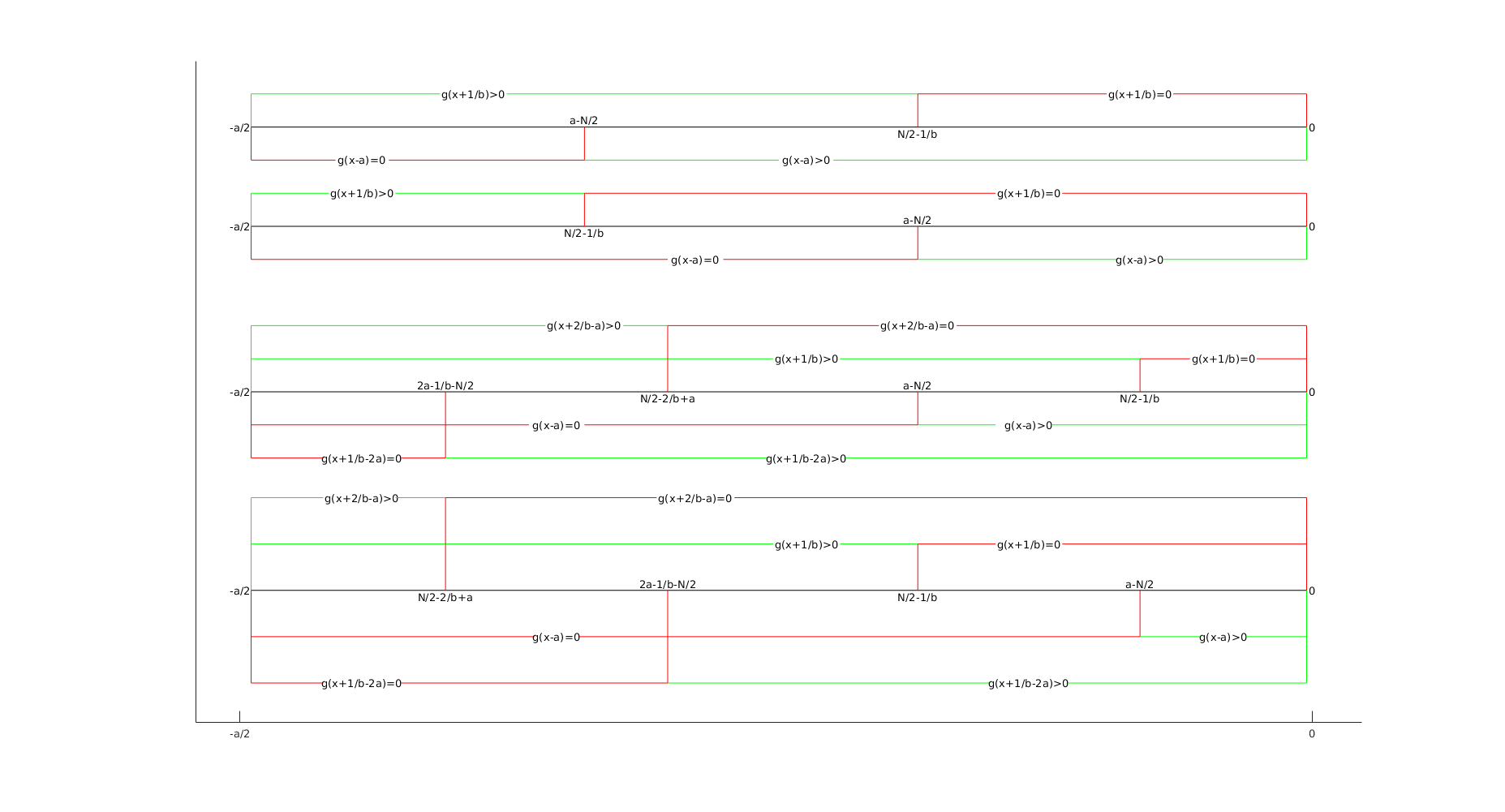}
 \caption{The off-diagonal entries of  $A_m(X)$ for  $m=2$ and $m=3$ when  $x\in [-\frac{a}{2}, 0].$ }
 \label{fig:figure3}
\end{figure}

\newpage

The next result relates the determinants $|G_m(x)|$ and $|A_m(x)|$.

\begin{cor}\label{cor:detgm} Given  $N\geq 2$ and $0<a<N$, suppose that  $g \in V_{N,a}$. Assume that $m\geq3$, and  let $(a,b)\in T_m$. Then for all $x\in [-a/2, 0]$, 
 \begin{equation*}
 |G_m(x)|=  |A_m(x)| \, \prod_{k=1-m}^{-1}{g(x+\frac{k}{b}-ka)}.
\end{equation*}

Furthermore, 
\item \begin{enumerate}
 \item[(a)] If $\frac{N}{2}-\frac{1}{b}\leq a-\frac{N}{2}$, then for all $x\in [-a/2, 0]$, 
\begin{equation*}
 |G_m(x)|=\prod_{k=1-m}^{m-1}{g(x+\frac{k}{b}-ka)}.
\end{equation*}

\item[(b)]  If $a-\frac{N}{2}<\frac{N}{2}-\frac{1}{b}$, then for all  $x\in S_m$, 
\begin{equation*}
 |G_m(x)|=\prod_{k=1-m}^{m-1}{g(x+\frac{k}{b}-ka)}, 
\end{equation*}
and for all $x\in Z_m$
\begin{equation*}
 |G_m(x)|= |A_{-\ell,-\ell-1}(x)|\, \prod_{\begin{array}{c}\scriptstyle k=1-m \\ \scriptstyle k\neq-\ell,-\ell-1\end{array}}^{m-1}{g(x+\frac{k}{b}-ka)},
\end{equation*} where the sets $S_m$ and $Z_m$ were defined in Lemma~\ref{lem:subdeter}.
\end{enumerate}

\end{cor}

\begin{proof} Recall that from Remark~\ref{rem:gmblock}, $G_m(x)$ can be written as a block matrix: for $x\in [-a/2, 0]$,
$$G_m(x)=\begin{bmatrix}A_m(x) & B_{m}(x)\\ 0 & C_m(x)\end{bmatrix}.$$ 

%
%

Now by computing the determinant of $G_m(x)$ using Laplace expansion by minors  along its last row, wee see that $$|G_m(x)|=\prod_{k=1-m}^{-1}{g(x+\frac{k}{b}-ka)}\cdot|A_m(x)|.$$
 
 The second part follows from Lemma~\ref{lem:subdeter}. 

\end{proof}

Finally, we can prove that the matrix  $G_m(x)$ is invertible for $x\in [-\frac{a}{2},\frac{a}{2}]$ under the assumptions of Theorem~\ref{th:main1}.

\begin{cor}\label{cor:invertgm} Given  $N\geq 2$ and $0<a<N$, suppose that  $g \in V_{N,a}$. Assume that $m\geq 1$, and  let $(a,b)\in T_m$.  Then, for all $x\in [-\frac{a}{2},\frac{a}{2}]$,  $|G_m(x)|>0.$  
\end{cor}

\begin{proof} Recall that this result is known for $m=1, 2$, see,  \cite{Lemniel} and \cite{Olehon}.

 Let $x\in [-\frac{a}{2},0]$. By Corollary~\ref{cor:detgm}, we have  
 \begin{equation*}
 |G_m(x)|=\prod_{k=1-m}^{-1}{g(x+\frac{k}{b}-ka)}|A_m(x)|.
\end{equation*}
From Lemma~\ref{lem:subdeter}, we know that the determinant $|A_m(x)|$ is a product of 
 $g(x+\frac{k}{b}-ka)>0$, \mbox{$ k\in \left\{0,..,m-1\right\}$}, and 
$|A_{k,k-1}(x)|>0$, where $k\in \left\{1,...,m-1\right\}$. By  Lemma~\ref{lem:diagonal} and Lemma~\ref{lem:submatrices}, we conclude that $|G_m(x)|>0$ for all $ x\in [-\frac{a}{2},0]$, and by symmetry this holds for all $ x\in [-\frac{a}{2},\frac{a}{2}]$.

\end{proof}

We are now ready to prove Theorem~\ref{th:main1}.

\begin{proof}{\bf Proof of Theorem~\ref{th:main1}}

By Corollary~\ref{cor:invertgm} we know that $G_m(x)$ is invertible. Let $h$ be defined on $\R$ as follows. For  $x\in \R\setminus [-\frac{2m-1}{2}a,\frac{2m-1}{2}a]$ let $h(x)=0$, and for $x\in \left[-\frac{2m-1}{2}a,\frac{2m-1}{2}a\right]$ let $h$ be defined by 
$$
\begin{pmatrix}
 h(x+(1-m)a) \\ \vdots \\ h(x)\\ \vdots \\ h(x+(m-1)a) \end{pmatrix}=G^{-1}_m(x)\begin{pmatrix}0\\\vdots\\ 0\\ b\\0\\ \vdots \\0\end{pmatrix}=b(G^{-1}_m(x))_m$$ where $(G^{-1}_m(x))_m$ is the $m^{th}$ column vector of the matrix $G^{-1}_m(x).$
 
 Let $x\in (-\tfrac{a}{2}, 0]$, then we can solve for $h(x)$ for $x\in (-\tfrac{a}{2}+ka, ka]$ where $k\in \{1-m, \hdots, m-1\}$. By Remark~\ref{rem:parityh} we know that $h$ is even, so we can define $h$ on the interval $ [-\frac{2m-1}{2}a,\frac{2m-1}{2}a]$ except at finitely many points. But because $|G_m(x)|> 0$ for all $x\in [-\frac{a}{2},\frac{a}{2}]$, we conclude that $|G_m(x)|^{-1}$ is a continuous, hence a bounded function on $[-a/2, a/2]$. Consequently,  $h$  is a compactly supported and bounded function for which $\cG(h,a,b)$ is a Bessel sequence. By construction,  it also follows that  $g$ and $h$ are dual windows. 
\end{proof}

Corollary~\ref{cor:main1} is now easily proved:

\begin{proof}{\bf Proof of Corollary~\ref{cor:main1}}

It follows from the Theorem~\ref{th:main1} because the $T_m$ form a partition of $T$. 
\end{proof}

\begin{rem}\label{rem:hdiscont}
We can show that the dual constructed in Theorem~\ref{th:main1} is discontinuous. Indeed, for $x\in [-\tfrac{a}{2}, 0]$ $h(x)$ can be computed using Crammer's rule. Because the matrix $G_m(x)$ is upper triangular block matrix, one sees that $$h(x)=\tfrac{b\, |\tilde{A}_m(x)| \prod_{k=1-m}^{-1}g(x+\tfrac{k}{b}-ka)}{|G_m(x)|}$$ where $\tilde{A}_m(x)$ is the $(m-1)\times (m-1)$ matrix obtained by deleting the last column and the last row of $A_m(x)$. From Corollary~\ref{cor:invertgm} we conclude that $h(x)>0$ for $x\in [-\tfrac{a}{2}, 0]$. By symmetry, we know $h(x)>0$ on $[0, \tfrac{a}{2}]$. 

Now, let $x \in (\tfrac{a}{2}, a)$, and $y\in (-\tfrac{a}{2}, 0)$ such that $x=y+a$. Using again Crammer's rule and the structure of $G_m(x)$ it can be seen that $$h(x)=h(y+a)=\tfrac{|A_m(x)|\, |\tilde{C}_m(x)|}{|G_m(x)|}=0$$ where $\tilde{C}_m(x)$ is the $(m-1)\times (m-1)$ matrix obtained by replacing the first column of $C_m(x)$ by the $0$ vector (this comes from replacing column $m+1$ of $G_m(x)$ by the vector $be_{m+1}$ where $e_{m+1}$ is the $(m+1)^{th}$ standard unit vector). 

Therefore, $$0=\lim_{x\to \tfrac{a}{2}^{+}}h(x)\neq \lim_{x\to \tfrac{a}{2}^{-}}h(x)>0.$$

\end{rem}
We conclude the paper with the graph of the dual window $h$ for $(a, b)\in \{(9/10, 8/9), (3/4, 35/36)\}\subset  T_3$.

\begin{figure}[!h]
\includegraphics[scale=0.35]{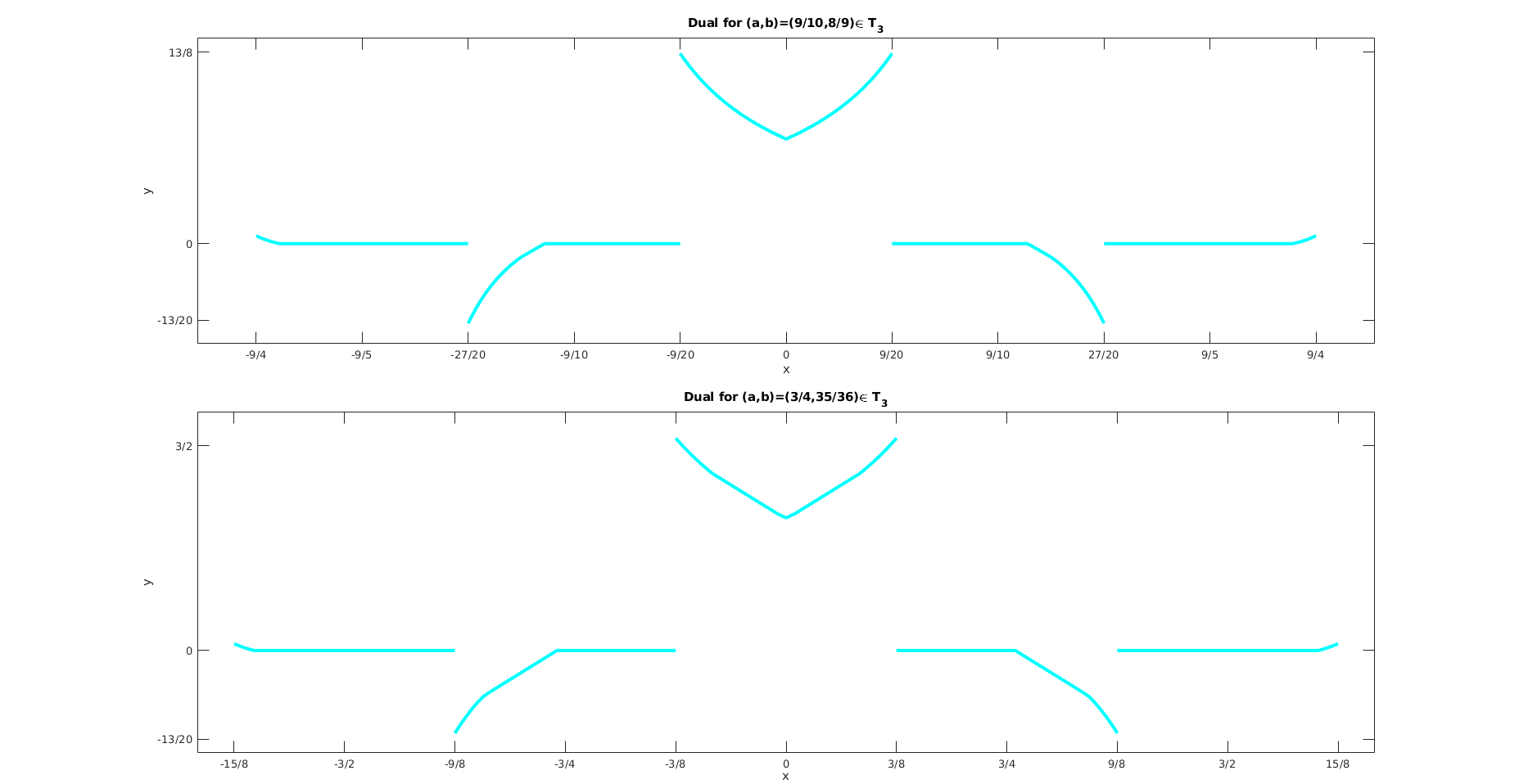}
 \caption{Graph of the dual window $h$ for $(a, b)\in \{(9/10, 8/9), (3/4, 35/36)\}\subset  T_3$ }
 \label{fig:figure4}
\end{figure}

\section*{Acknowledgements}
Part of this work was completed while the first-named author was a visiting graduate student in  the Department of Mathematics at the University of Maryland during the Fall 2017 semester. He would like to thank the Department for its hospitality and the African Center of Excellence in Mathematics and Application (CEA-SMA) at the  Institut de Math\'ematiques et de Sciences Physiques (IMSP) for funding his visit. K.~A.~Okoudjou  was partially supported by a grant from the Simons Foundation $\# 319197$, and by ARO grant W911NF1610008.



\bibliographystyle{amsplain}
\bibliography{splineFS}

\end{document}